\documentclass[a4paper,twoside,11pt]{amsart}  
\usepackage{amsmath, amsthm, amssymb} 
\usepackage{enumitem}
\usepackage{hyperref}
 \usepackage[shortalphabetic]{amsrefs}
\DeclareMathOperator*{\osc}{osc}

\newtheorem{theorem}{Theorem}[section] 
\newtheorem{proposition}[theorem]{Proposition}
\newtheorem{lemma}[theorem]{Lemma} 
\newtheorem{corollary}[theorem]{Corollary} 
\newtheorem*{gapconjecture}{Gap Conjecture}
 {\theoremstyle{remark}

}
\newcommand{\RR}{\mathbb{R}}
\newcommand{\NN}{\mathbb{N}}

\newcommand{\E}{\mathrm{e}}
\newcommand{\ep}{\varepsilon}

\begin{document} 
\title{Proof of the fundamental gap conjecture} 
%\date{\today}
\author{Ben Andrews}
\address{Centre for Mathematics and its Applications,  Australian National University, A.C.T. 0200, Australia}
\email{Ben.Andrews@maths.anu.edu.au}
\thanks{Research supported by Discovery Grant DP0985802 of the Australian Research Council}
\author{Julie Clutterbuck}
\address{Centre for Mathematics and its Applications,  Australian National University, A.C.T. 0200, Australia}
\email{Julie.Clutterbuck@maths.anu.edu.au}

\begin{abstract}  We prove the Fundamental Gap Conjecture, which states that the difference between the first two Dirichlet eigenvalues (the spectral gap) of a Schr\"odinger operator with convex potential and Dirichlet boundary data on a convex domain is bounded below by the spectral gap on an interval of the same diameter with zero potential.   More generally, for an arbitrary smooth potential in higher dimensions, our proof gives both a sharp lower bound for the spectral gap and a sharp modulus of concavity for the logarithm of the first eigenfunction, in terms of the diameter of the domain and a modulus of convexity for the potential.
\end{abstract}

\maketitle

 \section{Introduction}

We consider Schr\"odinger operators of the form $-\Delta+V$ with Dirichlet boundary conditions on a compact convex domain $\Omega$ in $\RR^n$.  The diameter of $\Omega$ is given by $D=\sup_{x,y\in\Omega} |y-x|$.  We assume the potential $V$ is semiconvex (i.e. $V+c|x|^2$ is convex for some $c$).  
Such an operator has an increasing sequence of eigenvalues $\lambda_0<\lambda_1\le \lambda_2\le\dots$ and corresponding eigenfunctions $\lbrace \phi_i \rbrace_{i\ge 0}$ which vanish on $\partial\Omega$ and satisfy the equation
\begin{align} \label{evalue problem 1} 
\Delta \phi_i - V\phi_i + \lambda_i \phi_i=0.
\end{align}

The difference between the first two eigenvalues, $\lambda_1-\lambda_0$, is called the \emph{fundamental gap}.  It is of importance for several reasons:   In quantum mechanics it represents the `excitation energy' required to reach the first excited state from the ground state; it thus determines the stability of the ground state and so is also important in statistical mechanics and quantum field theory.  The spectral gap also determines the rate at which positive solutions of the heat equation approach the first eigenspace, and it is through this characterization that we will prove the conjecture of the title.

It was observed by Michael van den Berg \cite{vandenBerg} that for many convex domains $\lambda_1-\lambda_0\ge 3\pi^2/D^2$.   This was also independently suggested by Ashbaugh and Benguria \cite{ashbaugh-benguria-89} and Yau \cite{MR865650}. 
\begin{gapconjecture}
Let $\Omega\subset\mathbb{R}^n$ be a bounded convex domain of diameter $D$, and $V$ a weakly convex potential.  
Then the eigenvalues of the Schr\"odinger operator
satisfy  
$$\lambda_2-\lambda_1\ge \frac{3\pi^2}{D^2}.$$
\end{gapconjecture}

 In one dimension, the conjecture is settled.     If $ \Omega$ is an interval of length $D$, the eigenvalue equation \eqref{evalue problem 1} becomes a boundary value problem for an ODE:
The eigenfunctions $\tilde{\phi}_i$ satisfy $\tilde\phi_i(\pm D/2)=0$ and
\begin{align} \label{evalue problem 2}
\tilde{\phi}_i''- \tilde V\tilde{\phi}_i +\mu_i\tilde{\phi}_i &= 0 \text{ in }(-D/2,D/2).  
\end{align}
For convenience we normalize these to have $\tilde\phi_i'(D/2)=-1$.

Ashbaugh and Benguria  \cite{ashbaugh-benguria-89} proved the conjecture under slightly different assumptions:   They considered symmetric single-well potentials, so that $\tilde{V}(x)=\tilde{V}(-x)$, and $\tilde{V}$ is non-decreasing on $[0,D/2]$.
 
\begin{theorem}[Ashbaugh--Benguria] \label{ash-ben} If $n=1$ and $\tilde{V}$ is symmetric and single-well, then the eigenvalue gap satisfies
$$\lambda_1-\lambda_0\ge \frac {3\pi^2}{D^2},$$
with equality attained when $\tilde{V}$ is constant.  
\end{theorem}

Later,  Horv\'ath \cite{MR1948113} partly removed the symmetry assumption, allowing $\tilde V$ to be a single-well potential with minimum at the mid-point of the interval.  The conjecture itself (with $\tilde V$ convex)  was proved by Lavine in 1994 \cite{Lavine-gap}.  

\begin{theorem}[Lavine] The Gap Conjecture holds for $n=1$.
\end{theorem}

%%%%%%%% to check:  relation between 1d and and problems clear in description.

In higher dimensions, the conjecture has been proved for some highly symmetric cases:  Ba{\~n}uelos--M{\'e}ndez-Hern{\'a}ndez \cite{mendez-banuelos} (for $V\equiv 0$), Davis \cite{MR1821082} and 
 Ba{\~n}uelos--Kr{\"o}ger \cite{kroger-banuelos} proved it for $\Omega\subset\mathbb{R}^2$ symmetric with respect to $x$ and $y$ axes, and convex in both $x$ and $y$.

Non-optimal lower bounds have also been obtained:  In an influential paper, Singer, Wong, Yau and Yau \cite{swyy-gap-estimate} used gradient estimates similar to those developed by Li \cite{Li} and Li--Yau \cite{Li-Yau}, together with the log-concavity of the first eigenfunction (proved for convex potentials by Brascamp and Lieb \cite{MR0450480}, and also by different methods in \cite{korevaar-convex} and in \cite{swyy-gap-estimate}) to show
 the gap is at least $\frac{\pi^2}{4D^2}$.     Yu and Zhong \cite{yu-zhong-gap} improved this to $\frac{\pi^2}{D^2}$, following the work of Zhong and Yang \cite{zhong-yang} which similarly improved Li and Yau's estimate for the first nontrivial eigenvalue on a compact manifold  \cite{Li-Yau}.   Many other authors have used probabilistic arguments (including \cite{MR0450480} and \cite{kroger-banuelos}).  Such methods were used by Smits \cite{smits} to recover the estimate of Yu and Zhong.
It seems that any further improvement requires an improved estimate on the log-concavity of the first eigenfunction.  Efforts in this direction have included work by  Yau \cite{yau-gap} and  Ling \cite{ling-gap}, who prove gap estimates depending explicitly on an upper bound for $D^2(\ln \phi_0)$.   
The problem of non-convex potentials has been studied by Yau \cite{yau-nonconvex-gap}.  
 There is also considerable literature on potentials with double wells, including examples with arbitrarily small spectral gap (see for example \cite{Harrell}).
 
We should point out here that fundamental gap for the Neumann problem (with zero potential) is simply the first nontrivial eigenvalue.  A lower bound for this on a convex domain was proved by Payne and Weinberger \cite{PW}.   Zhong and Yang \cite{zhong-yang} later reproduced the same result as a special case, using rather different methods.  

We refer to the excellent survey by Ashbaugh \cite{Ashbaughsurvey} for further discussion of the fundamental gap conjecture and its history.

Our proof of the gap conjecture combines two ingredients, both of which involve analysis of the heat equation:  The first ingredient is a method for estimating the modulus of continuity of solutions of parabolic equations which was developed in our earlier papers 
\cites{andrews-clutterbuck-jde,ac:gradient-estimates-hd}.  In this paper we say that a function $\eta$ on $\RR_+$ is a \emph{modulus of continuity}  for a function $f$ on $\Omega$ if
 \begin{equation}\label{eq:mod-of-cont}
 |f(y)-f(x)|\leq 2\eta\left(\frac{|y-x|}{2}\right)
 \end{equation}
 for all $x,y\in\Omega$.  Here the factors of 2 are for later convenience.  In \cite{ac:gradient-estimates-hd}*{Theorem 4.1} we proved that if $u$ is any smooth solution of the heat equation with Neumann boundary data on a convex domain $\Omega$ of diameter $D$ in $\RR^n$, then $u(.,t)$ has modulus of continuity $\varphi(.,t)$, where $\varphi$ is the solution to the one-dimensional heat equation on $[-D/2,D/2]$ with Neumann boundary data and initial data $\varphi(z,0)=\frac12\osc u(.,0)\text{sgn}(z)$.  The proof of this uses a maximum principle argument applied to a function of $2n$ spatial variables, and crucially exploits the full $2n\times 2n$ matrix of second derivatives. In that paper we were mostly concerned with the short-time application of this estimate to deduce Lipschitz bounds on $u$, but the estimate is also useful for large $t$:  One can show that $|\varphi(x,t)|\leq C\E^{-\frac{\pi^2}{D^2}t}$, so the oscillation of $u$ decays at an exponential rate, with exponent at least $\pi^2/D^2$.  But if we choose $u$ to be the first nontrivial eigenfunction for the Neumann problem on $\Omega$, then $u(x,t)=u(x,0)\E^{-\lambda t}$, where $\lambda$ is the corresponding eigenvalue.  It follows that $\lambda\geq \frac{\pi^2}{D^2}$, and we recover the sharp estimate of Zhong and Yang \cite{zhong-yang}, at least in the case of convex domains in $\RR^n$.  In a separate work we will apply this method to more general eigenvalue estimates on manifolds.
 
 As might be expected from the proof in \cite{swyy-gap-estimate}, a second ingredient is required to apply this technique to the gap conjecture:  The log-concavity of the first eigenfunction $\phi_0$.  If we apply the method above using only the log-concavity of $\phi_0$, then we recover the estimate
 $\lambda_1-\lambda_0\geq \frac{\pi^2}{D^2}$, which is precisely the estimate of Yu and Zhong \cite{yu-zhong-gap}.  As noted above, it is clear that we must use something stronger than log-concavity to obtain a sharp result.  Previous work \cites{yau-gap,ling-gap} produces stronger results assuming positive lower bounds on the Hessian of $\log\phi_0$, but such estimates cannot yield sharp results.
 
In trying to apply the modulus of continuity argument above, we are led to a log-concavity condition involving differences of derivatives of $\log\phi_0$. 
 
 We first introduce the relevant notions in a general context:
 If $\omega$ is a real function of a positive real variable, and $X$ is a vector field on a domain $\Omega$, we say $\omega$ is a \emph{modulus of expansion} for $X$ if for every $y\neq x$ in $\Omega$ we have
 \begin{equation}\label{mod of contraction}
 \left(X(y)-X(x)\right)\cdot\frac{y-x}{|y-x|}\geq 2\omega\left(\frac{|y-x|}{2}\right).
 \end{equation}
The factors of 2 are for later convenience.  The name reflects the fact that a modulus of expansion implies a rate of increase of distances under the flow of $X$. We say $\omega$ is a \emph{modulus of contraction} for $X$ if the sign is reversed (that is, $-\omega$ is a modulus of expansion for $-X$).  If $f$ is a semi-convex function on a domain $\Omega$, we say $\omega$ is a \emph{modulus of convexity} for $f$ if $\omega$ is a modulus of expansion for the gradient vector field $\nabla f$ of $f$, and $\omega$ is a \emph{modulus of concavity} for $f$ if $\omega$ is a modulus of contraction for $\nabla f$. Note $f$ is concave (convex) if and only if zero is a modulus of concavity (convexity) for $f$. 

Our log-concavity estimates amount to an explicit modulus of concavity for $\log\phi_0$.
Precisely, we assume that the potential $\tilde V$ is even, and that the potential $V$ in equation \eqref{evalue problem 1} is `more convex' than $\tilde V$, in the sense that for any $y\neq x$ in $\Omega$ we have 
\begin{align} \label{V-convexity}
\left( \nabla V (y)- \nabla V(x) \right)\cdot \frac{(y-x)}{|y-x|}\ge 2 \tilde V'\left(\frac{|y-x|}2\right).
\end{align}
That is, we assume that $\tilde V'$ is a modulus of convexity for $V$.  
Under this assumption we prove in Theorem \ref{log concavity} that the first eigenfunction $\phi_0$ in \eqref{evalue problem 1} is `more log-concave' than the first eigenfunction $\tilde\phi_0$ in \eqref{evalue problem 2}, in the sense that
\begin{equation} \label{log concavity eqn} 
\left(\nabla\ln \phi_0 (y)-\nabla\ln \phi_0(x)\right)\cdot \frac{y-x}{|y-x|}\le 2 (\ln \tilde\phi_0)'\big|_{z=\frac{|y-x|}2}
\end{equation}
for every $y\neq x$ in $\Omega$.  That is, we prove that $(\log\tilde\phi_0)'$ is a modulus of concavity for $\log\phi_0$.
Note that if $V$ is convex, then the inequality \eqref{V-convexity} holds with $\tilde V=0$, and the modulus of concavity in \eqref{log concavity eqn} becomes $-\frac{\pi}{D}\tan\left(\frac{\pi|y-x|}{D}\right)$.
We will now list the main results. 

\begin{theorem}  \label{spectral gap} 
If $V$ and $\tilde V$ are related by \eqref{V-convexity}, then the spectral gap for $-\Delta+V$ on a convex domain $\Omega$ is bounded below by the spectral gap of the one dimensional operator $-\frac{d^2}{dx^2}+\tilde V$ on $[-D/2,D/2]$.\end{theorem}
If $V$ is convex ,  the gap conjecture follows since we can choose $\tilde{V}=0$:

\begin{corollary} \label{convex corollary} If $V$ is weakly convex, then the spectral gap has the bound
$$\lambda_1-\lambda_0\ge \frac{3\pi^2}{D^2}.$$
\end{corollary}

In particular this applies for zero potential.
The key step in proving the gap conjecture is a sharp log-concavity estimate for the first eigenfunction:

\begin{theorem} 
\label{log concavity}  Let $\phi_0$ be the first eigenfunction of \eqref{evalue problem 1}.  If the potential $V$ satisfies condition \eqref{V-convexity}, then the log-concavity estimate \eqref{log concavity eqn} holds for every $y\neq x$ in $\Omega$.
\end{theorem}

Let us note a special case of this result which can be stated rather more elegantly:  Suppose 
$\tilde V'$ is also a modulus of convexity for $\tilde V$.  It then follows from Theorem \ref{log concavity} applied with $n=1$ and  $V=\tilde V$ that $(\log\tilde\phi_0)'$ is a (sharp) modulus of concavity for $\log\tilde\phi_0$.  This happens, for example, if $\tilde V$ has non-negative third derivatives for positive arguments:  In that case $\tilde V'(b)-\tilde V'(a)\geq \tilde V'(\frac{b-a}{2})-\tilde V'(-\frac{b-a}{2})=2\tilde V'(\frac{b-a}{2})$, so that $\tilde V$ has (sharp) modulus of convexity $\tilde V'$.  It also happens if $\tilde V$ is the optimal choice defined by $2\tilde V'(z/2)=\sup\{(\nabla V(x+ze)-\nabla V(x))\cdot e:\ x,x+ze\in\Omega,\ \|e\|=1\}$.  Thus in these cases we have:

\begin{corollary}\label{cor: compare modulus}
Under the conditions above, if the modulus of convexity of $V$ is bounded below by the modulus of convexity of $\tilde V$, then the modulus of concavity of $\log\phi_0$ is bounded above by the modulus of concavity of $\log\tilde\phi_0$.
\end{corollary}

Throughout the computations we assume that $\Omega$ has smooth boundary and is uniformly convex (so the principal curvatures of $\partial\Omega$ have a positive lower bound), and that $V$ is smooth.  The results for the general case of convex $\Omega$ and semiconvex $V$ follow using a straightforward approximation argument
(see for example \cite{MR2512810} and \cite{henrot-book}*{Th 2.3.17}).

The paper is arranged as follows:   In Section \ref{spectral gap section}  we prove a general modulus of continuity estimate for solutions of Neumann heat equations with drift.  In Section \ref{implication} we use this to prove
Theorem \ref{spectral gap}, assuming the result of Theorem \ref{log concavity}.  The proof of the modulus of continuity estimate provides a model for the more involved log-concavity estimate, which we
prove in Section \ref{log concavity section}.   Finally, in Section \ref{examples section} we discuss some examples and extensions of our results.

\section{Modulus of continuity for heat equation with drift} \label{spectral gap section}

We prove here a generalization of \cite{ac:gradient-estimates-hd}*{Theorem 4.1}, which controls the modulus of continuity of solutions of a Neumann heat equation with drift, in terms of the modulus of contraction of the drift velocity:

\begin{theorem}\label{thm: drift diffusion}
Let $\Omega$ be a strictly convex domain of diameter $D$ with smooth boundary in $\RR^n$, and $X$  a time-dependent vector field on $\Omega$.  Suppose $v:\ \Omega\times\RR_+\to\RR$ is a smooth solution
of the equation
\begin{align}\label{eq:diff_drift}
\frac{\partial v}{\partial t} &= \Delta v + X\cdot\nabla v\quad\text{in\ }\Omega\times\RR_+;\\
D_\nu v&=0\quad\text{in\ }\partial\Omega\times\RR_+.\notag
\end{align}
Suppose that
\begin{enumerate}[label={\arabic*.}, ref={\arabic*}]
\item\label{item1} $X(.,t)$ has modulus of contraction $\omega(.,t)$ for each $t\geq 0$, where $\omega:\ [0,D/2]\times\RR_+\to\RR$ is smooth;
\item\label{item2} $v(.,0)$ has modulus of continuity $\varphi_0$, where $\varphi_0:\ [0,D/2]\to\RR$ is smooth with $\varphi_0(0)=0$ and $\varphi_0'(z)>0$ for $0\leq z\leq D/2$;
\item\label{item3} $\varphi:\ [0,D/2]\times\RR_+\to\RR$ satisfies 
\begin{enumerate}[label={(\roman*).}, ref={(\roman*)}]
\item\label{itema} $\varphi(z,0) = \varphi_0(z)$ for each $z\in[0,D/2]$;
\item\label{itemb} $\frac{\partial \varphi}{\partial t}\geq \varphi''+\omega\varphi'$ on $[0,D/2]\times\RR_+$;
\item\label{itemc} $\varphi'>0$ on $[0,D/2]\times\RR_+$;
\item\label{itemd} $\varphi(0,t)\geq 0$ for each $t\geq 0$.
\end{enumerate}
\end{enumerate}
Then $\varphi(.,t)$ is a modulus of continuity for $v(.,t)$ for each $t\geq 0$.
\end{theorem}

\begin{proof}
The only difference from the situation in \cite{ac:gradient-estimates-hd} is the drift term in \ref{eq:diff_drift}.  For any $\varepsilon\geq 0$, define a function $Z_\ep$ on $\bar\Omega\times\bar\Omega\times\RR_+$ by
$$
Z_\ep(y,x,t) = v(y,t)-v(x,t)-2\varphi\left(\frac{|y-x|}{2},t\right)-\varepsilon\E^t.
$$
By assumption $Z_\ep(x,y,0)\leq -\varepsilon$ for every $x\neq y$ in $\Omega$, and $Z_\ep(x,x,t)\leq-\ep$ for every $x\in\Omega$ and $t\geq 0$.  We will prove for any $\ep>0$ that $Z_\ep$ is negative on $\Omega\times\Omega\times\RR_+$.  If this is not true, then there exists a first time $t_0>0$ and points $x\neq y\in\bar\Omega$ such that $Z_\ep(x,y,t_0)=0$.  

We consider two possibilities:  If $y\in\partial\Omega$, then we have 
\begin{equation*}
D_{\nu_y}Z_\ep=D_{\nu_y}v(y,t)-\varphi'\frac{(y-x)}{|y-x|}\cdot \nu_y 
=-\varphi'\frac{(y-x)}{|y-x|}\cdot \nu_y< 0, 
\end{equation*}
where $\nu_y$ is the outward unit normal at $y$, and we used the Neumann condition, strict convexity of $\Omega$ (which implies $\frac{y-x}{|y-x|}\cdot\nu_y>0$), and assumption \ref{item3}\ref{itemc}.  This implies $Z_\ep(x,y-s\nu_y,t_0)>0$ for $s$ small, contradicting the fact that $Z_\ep\leq 0$ on $\bar\Omega\times\bar\Omega\times[0,t_0]$. The case where $x\in\partial\Omega$ is similar.

The second possibility is that $x$ and $y$ are in the interior of $\Omega$.  Then all first spatial derivatives of $Z_\ep$ (in $x$ and $y$) at $(x,y,t_0)$ vanish, and the full $2n\times 2n$ matrix of second derivatives is non-positive.  In particular we choose an orthonormal basis with $e_n=\frac{y-x}{|y-x|}$, and will use the inequalities
\begin{align*}
\frac{\partial^2}{\partial s^2}Z_\ep(x+se_n,y-se_n,t_0)\big|_{s=0}&\leq 0;\quad\text{and}\\
\frac{\partial^2}{\partial s^2}Z_\ep(x+se_i,y+se_i,t_0)\big|_{s=0}&\leq 0\quad\text{for\ }i=1,\dots,n-1.
\end{align*}

We compute these inequalities in terms of $v$:  The vanishing of first derivatives implies 
$\nabla v(y)=\nabla v(x) = \varphi' e_n$.  Along the path $(x+se_i,y+se_i)$ the distance $|y-x|$ is constant, so  
$$
0\geq \frac{d^2}{ds^2}Z_\ep(x+se_i,y+se_i,t_0) = D_iD_iv(y)-D_iD_iv(x).
$$
Along $(x+se_n,y-se_n)$ we have $\frac{d}{ds}|y-x|=-2$ and $\frac{d^2}{ds^2}|y-x|=0$, so
$$
0\geq \frac{d^2}{ds^2}Z_\ep(x+se_n,y-se_n,t_0) = D_nD_nv(y)-D_nD_nv(x)-2\varphi''.
$$
Summing the first inequality over $i=1,\dots,n-1$ and adding the last gives
$$
0\geq \Delta v(y)-\Delta v(x)-2\varphi''.
$$
Finally we compute the time derivative of $Z_\ep$ at $(x,y,t_0)$:
\begin{align*}
\frac{\partial}{\partial t}Z_\ep &= \Delta v(y)+X(y)\cdot\nabla v(y)\\
					&\quad\null-\Delta v(x) -X(x)\cdot\nabla v(x) -2\varphi_t-\ep\E^t\\
					&\leq 2\varphi'' +\varphi'\left(X(y)-X(x)\right)\cdot\frac{y-x}{|y-x|}-2\varphi_t-\ep\E^t\\
					&<2\varphi''+2\omega\varphi'-2\varphi_t\\
					&\leq 0,
\end{align*}
where we used the modulus of contraction of $X$, and the non-negativity of $\varphi'$, as well as the differential inequality for $\varphi$ assumed in the Theorem.  But this strict inequality is impossible since this is the first time where $Z_\ep\geq 0$.  This contradiction proves that $Z_\ep<0$ for every $\ep>0$, and therefore $Z_0\leq 0$ which is the claim of the Theorem.
\end{proof}

\section{Improved log-concavity implies the gap conjecture}\label{implication}

We will deduce Theorem \ref{spectral gap} by applying Theorem \ref{thm: drift diffusion} in the case where $v$ is a ratio of solutions to the parabolic equation corresponding to the operator $\Delta+V$.  We control the modulus of continuity of such a ratio assuming a modulus of concavity for the logarithm of the denominator:

\begin{proposition}\label{prop: Neumann for v}
Let $u_1$ and $u_0$ be two smooth solutions to the parabolic Schr\"odinger equation
\begin{align}\label{eq:Par.Schrod}
\frac{\partial u}{\partial t} &=\Delta u - V u\quad\text{on\ }\Omega\times\RR_+;\\
u&=0\quad\text{on\ }\partial\Omega\times\RR_+,\notag
\end{align}
on a bounded strictly convex domain $\Omega$ with smooth boundary in $\RR^n$, with $u_0$ positive on the interior of $\Omega$.   Let $v(x,t)=\frac{u_1(x,t)}{u_0(x,t)}$.  Then $v$ is smooth on $\Omega\times\RR_+$, and satisfies the following Neumann heat equation with drift:
\begin{align}\label{eq:Neumann_for_v}
\frac{\partial v}{\partial t}-\Delta v-2\nabla\log u_0\cdot\nabla v=0&\quad\text{on\ }\Omega\times[0,\infty);\\
D_\nu v= 0&\quad\text{on\ }\partial\Omega\times[0,\infty).\notag
\end{align}
\end{proposition}

\begin{proof}
The argument is essentially that of \cite{yau-gap}*{Lemma 1.1} and \cite{swyy-gap-estimate}*{Appendix A}:
Both $u_0$ and $u_1$ are smooth on $\bar\Omega\times[0,\infty)$, and $u_0$ has negative derivative in the normal direction $\nu$.  It follows that $v$ extends to $\bar\Omega$ as a smooth function (see \cite{swyy-gap-estimate}*{Section 6}).  
We compute directly:
\begin{align}
\frac{\partial v}{\partial t} &=\frac{\partial}{\partial t}\left(\frac{u_1}{u_0}\right)\notag\\
&=\frac{\Delta u_1-Vu_1-v(\Delta u_0-Vu_0)}{u_0}\notag\\
&=\Delta v +2\nabla\log u_0\cdot\nabla v.\label{eq:v_evolution}
\end{align}
At any point in $\partial\Omega$ we have $\frac{\partial v}{\partial t}$ and $\Delta v$ bounded, $\nabla u_0=-c\nu$ with $c>0$, and $u_0=0$, so it follows from \eqref{eq:v_evolution} that $D_\nu v=0$.
\end{proof}

We now establish the fundamental gap conjecture, assuming Theorem \ref{log concavity}:

\begin{proposition}\label{prop:specgap}
Inequality \eqref{log concavity eqn} implies Theorem \ref{spectral gap}.
\end{proposition}

\begin{proof}
We use the strategy discussed in the introduction to deduce the eigenvalue bound from an exponentially decaying modulus of continuity bound, which we deduce from Theorem \ref{thm: drift diffusion} and Proposition \ref{prop: Neumann for v}.  

Suppose that $u_1$ is any smooth solution of \eqref{eq:Par.Schrod}, and let $u_0$ be the solution of \eqref{eq:Par.Schrod} with initial data $\phi_0$.  Let $v=\frac{u_1}{u_0}$ be the corresponding solution of \eqref{eq:Neumann_for_v}.  The drift velocity $X=2\nabla\log u_0$ in
equation \eqref{eq:Neumann_for_v} has modulus of contraction $2(\log\tilde\phi_0)'$, by the estimate
\eqref{log concavity eqn}.  

Roughly speaking the strategy is to apply Theorem \ref{thm: drift diffusion} with $\varphi=C\frac{\E^{-\mu_1t}\tilde\phi_1}{\E^{-\mu_0t}\tilde\phi_0}$ (that is, a ratio of particular solutions for the corresponding one-dimensional heat equation) and $\omega=2(\log\tilde\phi_0)'$.   Condition \ref{item2} of Theorem \ref{thm: drift diffusion} holds for sufficiently large $C$, since $v(.,0)$ is smooth, and so has bounded gradient and satisfies $|v(y,0)-v(x,0)|\leq K|y-x|$ for some large $K$, while $\frac{\tilde\phi_1}{\tilde\phi_0}$ is positive on $(0,D/2)$ and has positive gradient at $z=0$, so is bounded below by $cz$ for some small $c>0$.  The conditions \ref{item3}\ref{itema}--\ref{itemd} also hold, except that $\varphi'(D/2,t)=0$ and $\omega$ is not smooth at $D/2$.  We handle these complications below by making slight adjustments to $\varphi$ and $\omega$.

For small $\ep>0$ we construct a smooth function $\omega_\ep\geq\omega$ and corresponding functions $\varphi_{\ep,\tilde\ep}$ as follows:  We first replace the Dirichlet eigenfunction $\tilde\phi_0$ by a Robin eigenfunction $\tilde\phi_{0,\ep}$ with eigenvalue $\mu_{0,\ep}<\mu_0$ defined by
\begin{align}\label{eq:phi0ep}
(\tilde\phi_{0,\ep})''-\tilde V\tilde\phi_{0,\ep}+\mu_{0,\ep}\tilde\phi_{0,\ep}&=0\quad\text{on\ }[0,D/2];\\
\tilde\phi_{0,\ep}(D/2)=\ep;&\quad \tilde\phi_{0,\ep}'(D/2)=-1;\notag\\
\tilde\phi_{0,\ep}'(0)=0;\quad&
\tilde\phi_{0,\ep}>0\quad\text{on\ }[0,D/2].\notag
\end{align}
We also replace $\tilde\phi_1$ by $\tilde\phi_{1,\tilde\ep}$ with eigenvalue $\mu_{1,\tilde\ep}$ defined by
\begin{align}\label{eq:phi1ep}
(\tilde\phi_{1,\tilde\ep})''-\tilde V\tilde\phi_{1,\tilde\ep}+\mu_{1,\tilde\ep}\tilde\phi_{1,\tilde\ep}&=0\quad\text{on\ }[0,D/2];\\
\tilde\phi_{1,\tilde\ep}(D/2)=\tilde\ep;&\quad\tilde\phi_{1,\tilde\ep}'(D/2)=-1;\notag\\
\tilde\phi_{1,\tilde\ep}(0)=0;\quad&
\tilde\phi_{1,\tilde\ep}>0\quad\text{on\ }(0,D/2].\notag
\end{align}
These may be constructed using the Pr\"ufer transformation:  If $\phi''=(\tilde V-\mu)\phi$, then $q=\arctan(\phi'/\phi)$ satisfies a first order ODE.  Inverting this process we let $q(z,q_0,\mu)$ be the solution of
\begin{align}
\frac{\partial q}{\partial z} - (\tilde V-\mu)\cos^2q+\sin^2q&=0\quad |z|\leq D/2;\label{eq:qeqn}\\
q(0,q_0,\mu)&=q_0.\notag
\end{align}
The ODE comparison theorem implies $q$ is strictly increasing in $q_0$ for all $z$, and strictly decreasing in $\mu$ for $z>0$.  The choice $\mu=\mu_0$ and $q_0=0$ corresponds to $\tilde\phi_0$, and so $q(D/2,0,\mu_0)=-\pi/2$ and $q(z,0,\mu_0)\in(-\pi/2,\pi/2)$ for $0<z<D/2$.  Similarly, $\mu=\mu_1$ and $q_0=\pi/2$ corresponds to $\tilde\phi_1$, and so $q(D/2,\pi/2,\mu_1)=-\pi/2$ and $q(z,\pi/2,\mu_1)\in(-\pi/2,\pi/2)$ for $0<z<D/2$.

Since $q(D/2,0,\mu)$ is strictly decreasing in $\mu$ and $q(D/2,0,\mu_0)=-\pi/2$, for small $\ep>0$ there exists a unique $\mu_{0,\ep}<\mu_0$ such that $q(D/2,0,\mu_{0,\ep})=\arctan\ep-\pi/2$ and $q(z,0,\mu_{0,\ep})\in(-\pi/2,\pi/2)$ for $0<z<D/2$.  The corresponding eigenfunction
$\tilde\phi_{0,\ep}(z) = \ep\exp\left(-\int_z^{D/2} \tan q(s,0,\mu_{0,\ep})\,ds\right)$ satisfies conditions \eqref{eq:phi0ep}.  We have $(\log\tilde\phi_{0,\ep})'(z)=\tan q(z,0,\mu_{0,\ep})>\tan q(z,0,\mu_0) = (\log\tilde\phi_0)'(z)$ for $z\in(0,D/2]$, since $\mu$ was decreased.  Thus $\omega_\ep := 2(\log\tilde\phi_{0,\ep})'$ is smooth and $\omega_\ep\geq\omega$, so $\omega_\ep$ is a modulus of contraction for $X$.

The eigenfunction $\tilde\phi_{1,\tilde\ep}$ satisfying \eqref{eq:phi1ep} is similarly produced by decreasing $\mu$ below $\mu_1$ with $q_0=\pi/2$.

Observe that $(\log\tilde\phi_{1,\ep})'(z)=q(z,\pi/2,\mu_{1,\ep})$ satisfies the same ODE \eqref{eq:qeqn} as $(\log\tilde\phi_{0,\ep})'(z)=q(z,0,\mu_{0,\ep})$ with larger $\mu$, and they are equal at $z=D/2$.  By ODE comparison,  $(\log\tilde\phi_{1,\ep})'(z)>(\log\tilde\phi_{0,\ep})'(z)$ for $0\leq z<D/2$, and hence if $\tilde\ep>\ep>0$ then $(\log\tilde\phi_{1,\tilde\ep})'(z)>(\log\tilde\phi_{0,\ep})'(z)$ for $0\leq z\leq D/2$.

For any $\tilde\ep>\ep>0$ we define a smooth function $\varphi_{\ep,\tilde\ep}$ by
$$
\varphi_{\ep,\tilde\ep}=\frac{\E^{-\mu_{1,\tilde\ep}t}\tilde\phi_{1,\tilde\ep}}{\E^{-\mu_{0,\ep}t}\tilde\phi_{0,\ep}}.
$$
Since $\varphi_{\ep,\tilde\ep}$ is a ratio of solutions of the heat equation, we have
$$
\frac{\partial}{\partial t}\varphi_{\ep,\tilde\ep} = \varphi_{\ep,\tilde\ep}''+2(\log\tilde\phi_{0,\ep})'\varphi_{\ep,\tilde\ep}'
=\varphi_{\ep,\tilde\ep}''+\omega_\ep\varphi_{\ep,\tilde\ep}'.
$$
The monotonicity requirement \ref{item3}\ref{itemc} is satisfied, since
$$
\varphi'_{\ep,\tilde\ep} = \varphi_{\ep,\tilde\ep}\left((\log\tilde\phi_{1,\tilde\ep})'-(\log\tilde\phi_{0,\ep})'\right)>0
\quad\text{on\ }[0,D/2].
$$
Hence Theorem \ref{thm: drift diffusion} applies to prove that $C\varphi_{\ep,\tilde\ep}(.,t)$ is a modulus of continuity for $v$ for $t>0$ if it is so at $t=0$.  Sending $\tilde\ep$ to $\ep$ and then sending $\ep$ to zero, we have that $\varphi_{\ep,\tilde\ep}$ converges uniformly to $\varphi$, so the same conclusion holds for $\varphi$.  In particular, for sufficiently large $C$ we have
$$
\osc v(.,t)\leq 2C\sup\{\varphi(z,t):\ z\in[0,D/2]\} = 2C\E^{-(\mu_1-\mu_0)t}.
$$
Applying this with $u_1=\phi_1\E^{-\lambda_1t}$ gives $v=\E^{-(\lambda_1-\lambda_0)t}\frac{\phi_1}{\phi_0}$, and hence
$$
\E^{-(\lambda_1-\lambda_0)t}\osc\frac{\phi_1}{\phi_0}\leq 2C\E^{-(\mu_1-\mu_0)t}
$$
for all $t\geq 0$, which implies $\lambda_1-\lambda_0\geq \mu_1-\mu_0$.  This proves Theorem \ref{spectral gap}.
\end{proof}

\section{Improved log-concavity of the ground state} \label{log concavity section}

\newcommand{\phizero}{\phi} %first eigenfunction 

In this section we prove the log-concavity result, Theorem \ref{log concavity}.  We will deduce this from the following statement about log-concavity for arbitrary positive solutions of the parabolic equation \eqref{eq:Par.Schrod}:
   
\begin{theorem}\label{thm:log conc}
Let $\Omega$ be a uniformly convex open domain with diameter $D$ in $\RR^n$, and assume that $\tilde V'$ is a modulus of convexity for the potential $V$ on $\Omega$.  Let $u_0$ be a smooth function on $\bar\Omega$ which is positive on $\Omega$ and has $u_0=0$ and $\nabla u_0\neq 0$ on $\partial\Omega$, and suppose $\psi_0:\ [0,D/2]\to\RR$ is a Lipschitz continuous modulus of concavity for $\log u_0$.
Let $u:\ \Omega\times\RR_+\to\RR$ be the solution of the Dirichlet problem \eqref{eq:Par.Schrod} with initial data $u_0$, and let $\psi\in C^0([0,D/2]\times\RR_+)\cap C^\infty([0,D/2]\times(0,\infty))$ be a solution of
\begin{equation}\label{eq:dtpsi}
\left.\begin{aligned}
\frac{\partial\psi}{\partial t}\geq \psi''+2\psi\psi'-\tilde V'&\quad\text{on\ }[0,D/2]\times\RR_+\\
\psi(.,0)=\psi_0(.)&\\
\psi(0,t)=0&\\
\end{aligned}\right\}
\end{equation}
Then $\psi(.,t)$ is a modulus of concavity for $\log u(.,t)$ for each $t\geq 0$.
\end{theorem}

Note that we impose no boundary condition on $\psi$ at $(D/2,t)$.

\begin{proof}
The statement we must prove is the following:
$$
Z(x,y,t):=\left(\nabla\log u(y,t)-\nabla\log u(x,t)\right)\cdot\frac{y-x}{|y-x|}-2\psi\left(\frac{|y-x|}{2},t\right)\leq 0.
$$
for all $x\neq y\in\Omega$ and $t\geq 0$.  By assumption this holds for $t=0$.   We will prove this by showing negativity of a function
 $Z_\ep$ on $\hat\Omega\times[0,T]$ for any $\ep>0$ and $T\in(0,\infty)$, where $\hat\Omega=(\Omega\times\Omega)\setminus\{y=x\}\subset\RR^{2n}$, and $Z_\ep$ is defined by
\begin{align*}
Z_\ep(x,y,t) &= Z(x,y,t)-\ep\E^{Ct}\\
&=\left(\nabla\log u(y)-\nabla\log u(x)\right)\cdot\frac{y-x}{|y-x|}- 2\psi\left(\frac{|y-x|}{2},t\right)-\ep\E^{Ct},
\end{align*}
for some suitably large $C$ to be chosen (independent of $\ep$), where we have suppressed the $t$ dependence of $u$.  
By the Hopf boundary point lemma, $D_\nu u(x,t)<0$ for every $x\in\partial\Omega$ and every $t\geq 0$. 
The following Lemmas control the behaviour of $Z$ near the boundary:

\begin{lemma}\label{lem:logubdry}
Let $u$ be as in Theorem \ref{thm:log conc}, and $T<\infty$.  Then there exists $r_1>0$ such that $D^2\log u\big|_{(x,t)}<0$ whenever $d(x,\partial\Omega)<r_1$ and $t\in[0,T]$, and $N\in\RR$ such that $D^2\log u\big|_{(x,t)}(v,v)\leq N|v|^2$ for all $x\in\Omega$ and $t\in[0,T]$.
\end{lemma}

\begin{proof}
We first construct $U_T$.  Let $\alpha=\inf_{\partial\Omega\times[0,T]}|Du|>0$, and let $K$ be a bound for $D^2u$, so that $|D^2u(v,v)|\leq K|v|^2$ at every point in $\bar\Omega\times[0,T]$ and for all $v\in\RR^n$.  If $x_0\in\partial\Omega$ then 
$Du(x_0)=-|Du(x_0)|\nu(x_0)$, and $D^2u\big|_{x_0}(v,v)=h(v,v)Du\big|_{x_0}(\nu(x_0))$ for $v$ tangent to $\partial\Omega$, where $h$ is the second fundamental form of $\partial\Omega$ at $x_0$.  Uniform convexity implies $h(v,v)\geq\kappa|v|^2$ for some $\kappa>0$.  The gradient direction $e=\frac{Du(x)}{|Du(x)|}$ is smooth near $x_0$, as is the projection $\pi^\perp:\ w\mapsto (w\cdot e)e$ and the orthogonal projection $\pi=\text{Id}-\pi^\perp$.  At $x_0$ we have $D^2u(\pi w,\pi w)\leq -\alpha\kappa|\pi w|^2$.  Therefore there exists $r_0>0$ depending on $\alpha$ and $K$ such that for $x\in B_{r_0}(x_0)\cap\Omega$ and $t\in[0,T]$ we have
\begin{align*}
D^2u\big|_x(\pi w,\pi w)&\leq -\frac{\alpha\kappa}{2}|\pi w|^2\quad\text{for\ any\ }w\in\RR^n;\\
 |Du(x)|&\geq\frac{|Du(x_0)|}{2};\quad\text{and}\\
 0<u(x)&\leq 2|Du(x_0)||x-x_0|.
 \end{align*}
Therefore in this set and for any $w$ we have,
\begin{align*}
D^2u(w,w) &= D^2u(\pi w+\pi^\perp w,\pi w+ \pi^\perp w)\\
&\leq -\frac12\alpha\kappa|\pi w|^2 + 2K|\pi w||\pi^\perp w| + K|\pi^\perp w|^2\\
&\leq -\frac14\alpha\kappa|\pi w|^2 + (K+\frac{4K^2}{\alpha\kappa})|\pi^\perp w|^2.
\end{align*}
Since $(Du(w))^2=|Du|^2|\pi^\perp w|^2\geq \frac{|Du(x_0)|\alpha}{4}|\pi^\perp w|^2$ and $u\leq 2|Du(x_0)||x-x_0|$, 
\begin{align*}
D^2\log u\big|_x(w,w) &= \frac{1}{u}\left(D^2u(w,w)-\frac{(Du(w))^2}{u}\right)\\
&\leq \frac{1}{u}\left(-\frac14\alpha\kappa|\pi w|^2 + (K+\frac{4K^2}{\alpha\kappa})|\pi^\perp w|^2\right.\\
&\quad\quad\quad\null\left.-\frac{\alpha}{8|x-x_0|}|\pi^\perp w|^2\right)\\
&<0
\end{align*}
provided $|x-x_0|<r_1=\min\left\{r_0,\frac{\alpha^2\kappa}{8(K\alpha\kappa+4K^2)}\right\}$.   the Lemma holds with $N=\max\{0,\sup\{D^2\log u(z,t)(v,v):\ |v|=1,\ t\in[0,T],\ d(z,\partial\Omega)\geq r_1\}\}$, since 
$\{z\in\Omega:\ d(z,\partial\Omega)\geq r_1\}$ is compact. \end{proof}

\begin{lemma}\label{lem:bdry}
Let $u$ be as in Theorem \ref{thm:log conc}, and let $\psi$ be continuous on $[0,D/2]\times\RR_+$ and locally Lipschitz in the first argument, with $\psi(0,t)=0$ for each $t$.  Then 
for any $T<\infty$ and $\beta>0$ there exists an open set $U_{\beta,T}\subset\RR^{2n}$ containing $\partial\hat\Omega$ such that $Z(x,y,t)<\beta$ for all $t\in[0,T]$ and $(x,y)\in U_{\beta,T}\cap\hat\Omega$.
\end{lemma}

\begin{proof}
Since $\psi$ is Lipschitz in the first argument there exists $P$ such that $|\psi(z,t)|\leq Pz$ for all $z\in[0,D/2]$ and $t\in[0,T]$.

We construct $U_{\beta,T}$ of the form $\cup_{(x_0,y_0)\in\partial\hat\Omega}B_{r(x_0,y_0)}(x_0,y_0)$.  To do this we must find $r(x_0,y_0)>0$ for any $(x_0,y_0)\in\partial\hat\Omega$ such that $Z<\beta$ for $(x,y)\in B_{r(x_0,y_0)}(x_0,y_0)$.
We consider several cases for $(x_0,y_0)$:
\begin{enumerate}[label={{\bf Case \arabic*:}}, ref={case \arabic*}]
\item\label{case1} $x_0=y_0\in\Omega$.
We have by Lemma \ref{lem:logubdry}
\begin{align*}
Z(x,y,t) &= \left(\nabla\log u(y,t)-\nabla\log u(x,t)\right)\cdot\frac{y-x}{|y-x|}-2\psi\left(\frac{|y-x|}{2},t\right)\\
&=\frac{1}{|y-x|}\int_0^1 D^2\log u\big|_{(1-s)x+sy}(y-x,y-x)\,ds-2\psi\left(\frac{|y-x|}{2},t\right)\\
&\leq (N+P)|y-x|\end{align*}
so $Z<\beta$ provided $(x,y)\in B_r(x_0,x_0)$ with $r<\frac{\beta}{2(N+P)}$.
\item\label{case2} $x_0\in\partial\Omega$, $y_0\in\Omega$.  In this case $u(y_0)>0$, so there exists $r_2>0$ and $L>0$ such that $|D\log u(y)|\leq L$ for $|y-y_0|<r_2$.  Let $\alpha_0=|Du(x_0)|>0$ and $\gamma=-\frac{y_0-x_0}{|y_0-x_0|}\cdot \nu(x_0)$ (this is positive by strict convexity).
Then $Du(x_0)=-\alpha_0\nu(x_0)$, so $Du(x_0)\cdot \frac{y_0-x_0}{|y_0-x_0|}= \gamma\alpha_0$.  In addition, $Du(x)\cdot \frac{y-x}{|y-x|}$ is smooth in $x$ and $y$ near $(x_0,y_0)$.  Therefore $Du(x)\cdot\frac{y-x}{|y-x|}\geq\frac12\gamma\alpha_0$ and $0<u(x)\leq 2\alpha_0 |x-x_0|$ for $x,y\in\Omega$ with $\max\{|y-y_0|,|x-x_0|\}<r_3$, for some $r_3\in(0,r_2]$. This implies
\begin{align*}
Z(x,y,t) &= \left(\nabla\log u(y,t)-\nabla\log u(x,t)\right)\cdot\frac{y-x}{|y-x|}-2\psi\left(\frac{|y-x|}{2},t\right)\\
&\leq L-\frac{\gamma}{4|x-x_0|}+PD\end{align*}
so $Z<0$ provided $|y-y_0|<r_3$ and $|x-x_0|<\min\{r_3,\frac{\gamma}{4(L+PD)}\}$.
\item $x_0\in\Omega$, $y_0\in\partial\Omega$.  This is similar to \ref{case2}.
\item $x_0,y_0\in\partial\Omega$, $x_0\neq y_0$.  Here both $x$ and $y$ can be handled in the same way as $x$ in \ref{case2}.
\item $x_0=y_0\in\partial\Omega$.  Then letting $z_s = (1-s)x+sy$ we write as in \ref{case1}
\begin{equation*}
Z(x,y,t) =\frac{1}{|y-x|}\int_0^1 D^2\log u\big|_{z_s}(y-x,y-x)\,ds-2\psi\left(\frac{|y-x|}{2},t\right)
\end{equation*}
 The first term is negative for $\max\{|x-x_0|,|y-x_0|\}<r_1$ by Lemma \ref{lem:logubdry}, so
$Z(x,y,t)<\beta$ provided $\max\{|x-x_0|,|y-x_0|\}<\min\left\{r_1,\frac{\beta}{P}\right\}$.
\end{enumerate}
\end{proof}

Now we proceed with the proof of Theorem \ref{thm:log conc}.  Fix $T<\infty$ and $\ep>0$.  By assumption we have $Z_\ep(x,y,0)\leq -\ep\E^{Ct}<0$ for all $x\neq y$ in $\Omega$.  Off the compact set $(\hat\Omega\setminus U_{\ep/2,T})\times[0,T]$ we have $Z_\ep\leq -\frac12\ep$, by Lemma \ref{lem:bdry}.  Therefore if $Z_\ep$ is not negative on $\hat\Omega\times[0,T]$ then there exists a first time $t_0>0$ and points $(x,y)\in\hat\Omega\setminus U_{\ep/2,T}$ such that $Z_\ep< 0$ on $\hat\Omega\times[0,t_0)$, but $Z_\ep(x,y,t_0)=0$.  In particular $x$ and $y$ are in the interior of $\Omega$, and $y\neq x$.  We choose an orthonormal basis $\{e_i\}$ for $\RR^n$ with $e_n=\frac{y-x}{|y-x|}$, and note the following:  The first spatial derivatives of $Z_\ep$ in $x$ and $y$ vanish, so 
\begin{subequations}
\begin{align}
0&= \frac{\partial}{\partial s}Z_\ep(x+se_i,y,t_0)\big|_{s=0},\quad 1\leq i<n\notag\\
&=-D_iD_n\log u(x)-\frac{D_i\log u(y)-D_i\log u(x)}{|y-x|};\quad\text{and}\label{eq:D1Zxi}\\
0&=\frac{\partial}{\partial s}Z_\ep(x,y+se_i,t_0)\big|_{s=0},\quad 1\leq i<n\notag\\
&=D_iD_n\log u(y)+\frac{D_i\log u(y)-D_i\log u(x)}{|y-x|};\label{eq:D1Zyi}
\end{align}
while
\begin{align}
0&=\frac{\partial}{\partial s}Z_\ep(x+se_n,y,t_0)\big|_{s=0}=-D_nD_n\log u(x)+\psi';\label{eq:D1Zxn}\\
0&=\frac{\partial}{\partial s}Z_\ep(x,y+se_n,t_0)\big|_{s=0}=D_nD_n\log u(y)-\psi'.\label{eq:D1Zyn}
\end{align}
\end{subequations}
The second derivative of $Z_\ep(.,.,t_0)$ along any path is non-positive at $(x,y)$.  In particular $y-x$ is constant along $(x+se_i,y+se_i)$ for $1\leq i<n$, so
\begin{subequations}
\begin{align}\label{eq:D2Zii}
0&\geq \frac{\partial^2}{\partial s^2}Z_\ep(x+se_i,y+se_i,t_0)\big|_{s=0}\notag\\
&=D_iD_iD_n\log u(y)-D_iD_iD_n\log u(x).
\end{align}
Along the path $(x+se_n, y-se_n)$ we have $\frac{y-x}{|y-x|}$ constant, $\frac{d}{ds}|y-x|=-2$, and $\frac{d^2}{ds^2}|y-x|=0$.  This gives
\begin{align}\label{eq:D2Znn}
0&\geq \frac{\partial^2}{\partial s^2}Z_\ep(x+se_n,y-se_n,t_0)\big|_{s=0}\notag\\
&=D_nD_nD_n\log u(y)-D_nD_nD_n\log u(x)-2\psi''.
\end{align}
\end{subequations}
Finally, since $Z_\ep(x,y,t)< Z_\ep(x,y,t_0)$ for $t<t_0$, we have $\frac{\partial Z_\ep}{\partial t}(x,y,t_0)\geq 0$. 
We have
$
\frac{\partial}{\partial t}\log u = \Delta\log u + |D\log u|^2 -V
$ by \eqref{eq:Par.Schrod}, and
differentiation gives
$$
\frac{\partial}{\partial t}\nabla\log u = \Delta\nabla\log u + 2D_k\log uD_k\nabla\log u-\nabla V.
$$
This gives the following inequality:
\begin{align}
0&\leq\frac{\partial}{\partial t}Z_\ep\notag\\
&= \left(\Delta\nabla\log u(y)+2D_k\log uD_k\nabla\log u(y)-\nabla V(y)\phantom{\frac{y-x}{|y-x|}}\right.\notag\\
&\quad\null\left.\phantom{\frac{y}{y}}-\Delta\nabla\log u(x)-2D_k\log uD_k\nabla\log u(x)+\nabla V(x)\right)\cdot\frac{y-x}{|y-x|}\notag\\
&\quad\null-2\frac{\partial\psi}{\partial t}\notag-C\ep\E^{Ct}\\
&=D_nD_nD_n\log u(y)-D_nD_nD_n\log u(x)\notag\\
&\quad\null+\sum_{i=1}^{n-1}\left(D_iD_iD_n\log u(y)-D_iD_iD_n\log u(x)\right)\notag\\
&\quad\null + 2\left(D_n\log u(y)D_nD_n\log u(y)-D_n\log u(x)D_nD_n\log u(x)\right)\notag\\
&\quad\null+2\sum_{i=1}^{n-2}\left(D_i\log u(y)D_iD_n\log u(y)-D_i\log u(x)D_iD_n\log u(x)\right)\notag\\
&\quad\null-\left(\nabla V(y)-\nabla V(x)\right)\cdot\frac{y-x}{|y-x|}-2\frac{\partial\psi}{\partial t}-C\ep\E^{Ct}\notag
\end{align}
In the first line we use the second derivative inequality \eqref{eq:D2Znn}, and in the second we use \eqref{eq:D2Zii}.  In the third line we rewrite the second derivatives of $\log u$ using \eqref{eq:D1Zxn} and \eqref{eq:D1Zyn}, while in the fourth we use \eqref{eq:D1Zxi} and \eqref{eq:D1Zyi}.  In the last line we use the modulus of convexity assumption \eqref{V-convexity}.  This yields
\begin{align}
0&\leq 2\psi''+2\psi'\left(D_n\log u(y)-D_n\log u(x)\right)\notag\\
&\quad\null-\frac{2}{|y-x|}\sum_{i=1}^{n-1}\left|D_i\log u(y)-D_i\log u(x)\right|^2-2\tilde V'-2\frac{\partial\psi}{\partial t}-C\ep\E^{Ct}\notag.\\
&\leq 2\psi'' +2\psi'(2\psi+\ep\E^{Ct})-2\tilde V'-2\frac{\partial\psi}{\partial t}-C\ep\E^{Ct}\notag\\
&< 2\left(\psi''+2\psi\psi'-\tilde V'-\frac{\partial\psi}{\partial t}\right)\leq 0.\notag
\end{align}
Here the second inequality is obtained by discarding the negative term involving $D_i\log u$, and using $Z_\ep=D_n\log u(y)-D_n\log u(y)-2\psi-\ep\E^{Ct}=0$. The last inequality is from equation \eqref{eq:dtpsi}, and the strict inequality is produced by choosing $C>2\sup_{[0,D/2]\times[0,T]}\psi'$ (this is independent of $\ep$ as required).   The result contradicts the assumption that $Z_\ep$ does not remain negative.  Therefore $Z_\ep<0$ on $\hat\Omega\times[0,T]$ for every $\ep>0$ and every $T<\infty$.  It follows that $Z\leq 0$ on $\hat\Omega\times\RR_+$, and the Theorem is proved.
\end{proof}

We next consider the long term behaviour of solutions of \eqref{eq:dtpsi}:
\begin{corollary}\label{cor:lim.beh}
Under the conditions of Theorem \eqref{thm:log conc}, there exists a smooth function $\psi:\ [0,D/2)\times\RR_+$, decreasing in the second argument, such that 
$$
\left(\nabla\ln u(y,t)-\nabla\ln u(x,t)\right)\cdot \frac{y-x}{|y-x|}\leq 2\psi\left(\frac{|y-x|}{2},t\right)
$$
for all $x,y\in\Omega$ with $x\neq y$ and all $t\geq 0$, and $\lim_{t\to\infty}\psi(z,t)=(\log\tilde\phi_0)'(z)$ for every $z\in[0,D/2)$.
\end{corollary}

\begin{proof}
We construct a sequence of functions $\psi_k$ satisfying the assumptions of Theorem \ref{thm:log conc} and decreasing monotonically towards the (non-smooth) function $\psi$ as $k\to\infty$.  These are constructed by solving the equality case of \eqref{eq:dtpsi} with $\psi(D/2,t)=-k$ for suitable initial data.

The first step is to construct suitable initial functions $\psi_{k,0}$.   To do this we observe that there are many stationary solutions, corresponding to solutions of the eigenvalue equation via the Pr\"ufer transformation:  Stationary solutions satisfy
$0 = \psi''+2\psi\psi'-\tilde V'= (\psi'+\psi^2-\tilde V)'$,
and hence 
\begin{equation}\label{eq:Prufer}
\psi'+\psi^2=\tilde V-\mu
\end{equation}
for some $\mu$.  We consider solutions $\psi^L_\mu$ and $\psi^R_{k,\mu}$ of \eqref{eq:Prufer} satisfying $\psi^L_\mu(0)=0$ and $\psi^R_{k,\mu}(D/2)=-k$.  These relate to the Robin eigenfunctions $\tilde\phi_{0,\ep}$ satisfying equation \eqref{eq:phi0ep} which were constructed in the proof of Proposition \ref{prop:specgap}, with $\ep=1/k$:  In particular $\psi^L_{\mu_{0,1/k}}=\psi^R_{k,\mu_{0,1/k}}=(\log\tilde\phi_{0,1/k})'$.  The ODE comparison theorem applied to equation \eqref{eq:Prufer} implies that $\psi^L_{\mu}$ is strictly decreasing in $\mu$ for $0<z\leq D/2$, and $\psi^R_{k,\mu}$ is strictly increasing in $\mu$ for $0\leq z<D/2$.  In particular, $\psi^R_{k,\mu}>(\log\tilde\phi_{0,1/k})'$ for $\mu>\mu_{0,1/k}$ and $0\leq z<D/2$, and $\psi^L_{\mu}>(\log\tilde\phi_{0,1/k})'$ for $\mu<\mu_{0,1/k}$ and $0<z\leq D/2$. 
ODE comparison also gives upper bounds:  If $\lambda_+^2\geq\sup\tilde V-\mu$ and $\lambda_-^2\geq \mu-\inf\tilde V$ then
\begin{align*}
\psi^L_\mu(z)&\leq \lambda_+\tanh(\lambda_+ z);\qquad\text{and}\\
\psi^R_{k,\mu}(z)&\leq \frac{\lambda_-\tan(\lambda_-(D/2-z))-k}{1+\frac{k}{\lambda_-}\tan(\lambda_-(D/2-z))},\text{\ if\ }z>\frac{D}{2}-\frac{\frac{\pi}{2}+\arctan\left(\frac{k}{\lambda_-}\right)}{\lambda_-}.
\end{align*}
The upper and lower bounds imply existence of a supersolution
$\psi^+_{k,s}=\min\{\psi^L_{\mu_{0,1/k}-s},\psi^R_{k,\mu_{0,1/k}+s}\}$ for any $s\geq 0$.  Comparison gives lower bounds on these for large $s$:  For $s>\max\{\mu_{0,1/k}-\inf\tilde V,\sup\tilde V-\mu_{0,1/k}\}$,
\begin{equation}\label{eq:lowerbdbarrier}
\psi^+_{k,s}(z)\geq \begin{cases}
\tilde\lambda_+\tanh(\tilde\lambda_+z), &0\leq z\leq z_0;\\
\frac{\tilde\lambda_-\tan(\tilde\lambda_-(D/2-z))-k}{1+\frac{k}{\tilde\lambda_-}\tan(\tilde\lambda_-(D/2-z))}, &z_0\leq z\leq \frac{D}{2}
\end{cases}
\end{equation}
where $\tilde\lambda_+=\sqrt{s+\inf\tilde V-\mu_{0,1/k}}$, $\tilde\lambda_-=\sqrt{s+\mu_{0,1/k}-\sup\tilde V}$, and we take $z_0>\frac{D}{2}-\tilde\lambda_-^{-1}\left(\frac{\pi}{2}+\arctan\left(\frac{k}{\tilde\lambda_-}\right)\right)$ to make the two cases equal.

\begin{lemma}
For $u_0$ as in Theorem \ref{thm:log conc},
for each $k$ there exists $s(k)\geq 0$ such that 
$\psi^+_{k,s}$ is a modulus of concavity for $\log u_0$.
\end{lemma}

\begin{proof}
By Lemma \ref{lem:logubdry}, we have for all $(y,x)\in\hat\Omega$
\begin{equation*}
\left(\nabla\log u_0(y)-\nabla\log u_0(x)\right)\cdot\frac{y-x}{|y-x|}
\leq N|y-x|\leq 2\lambda_1\tanh\left(\frac{\lambda_1|y-x|}{2}\right),
\end{equation*}
where $\lambda_1$ is such that $ND\leq 2\lambda_1\tanh(\lambda_1 D/2)$.
Now apply Lemma \ref{lem:bdry} with $\psi(z)=-\frac{6kz}{D}$ and $\beta=k$ to find an open set $U\subset\RR^2$ containing $\partial\hat\Omega$ such that
$$
\left(\nabla\log u_0(y)-\nabla\log u_0(x)\right)\cdot\frac{y-x}{|y-x|}\leq -\frac{6k|y-x|}{D}+k
$$
for $(x,y)\in\hat\Omega\cap U$.  In particular there exists $\delta>0$ such that $U$ contains all points $(x,y)\in\hat\Omega$ with $|y-x|\geq D-\delta$.  Decreasing $\delta$ so that $\delta<D/3$ if necessary, we have that for all $(x,y)\in\hat\Omega$ with $|y-x|\geq D-\delta$,
\begin{align*}
\left(\nabla\log u_0(y)-\nabla\log u_0(x)\right)\cdot\frac{y-x}{|y-x|}&\leq -\frac{6k(D-\delta)}{D}+k\\
&\leq-2k\\
&\leq 2\frac{\lambda_2\tan\left(\lambda_2\left(\frac{D-|y-x|}{2}\right)\right)-k}{1+\frac{k}{\lambda_2}\tan(\lambda_2(\frac{D-|y-x|}{2}))}
\end{align*}
for any $\lambda_2>0$, provided $\frac{D-|y-x|}{2}<\frac{\frac{\pi}{2}+\arctan\left(\frac{k}{\lambda_2}\right)}{\lambda_2}$.  We choose $\lambda_2$ so large that
$\frac{\pi}{2}+\arctan\left(\frac{k}{\lambda_2}\right)<\delta\lambda_2$.  Then by \eqref{eq:lowerbdbarrier} we have for all $(x,y)\in\hat\Omega$
$$
\left(\nabla\log u_0(y)-\nabla\log u_0(x)\right)\cdot\frac{y-x}{|y-x|}\leq \psi^+_{k,s}\left(\frac{|y-x|}{2}\right)
$$
 provided $s\geq \max\{\lambda_1^2+\mu_{0,1/k}-\inf\tilde V,\lambda_2^2+\sup\tilde V-\mu_{0,1/k}\}$.
\end{proof}

Now we proceed to construct the solutions $\psi_k$:  We define
$$
s(k)=\inf\{s\geq 0:\ \psi^+_{k,s}\text{\ is\ a\ modulus\ of\ concavity\ for\ }\log u_0\}.
$$
Then we choose the initial data as follows:
$$
\psi_{k,0} = \min\left\{\psi^+_{j,s(j)}:\ 1\leq j\leq k\right\}.
$$
Note that $\psi_{k,0}$ is pointwise non-increasing in $k$.  Now define $\psi_k(z,t)$ by
\begin{equation}\label{eq:dtpsik}
\left.\begin{aligned}
\frac{\partial\psi_k}{\partial t}= \psi_k''+2\psi_k\psi_k'-\tilde V'&\quad\text{on\ }[0,D/2]\times(0,\infty)\\
\psi_k(.,0)=\psi_{k,0}(.)&\\
\psi_k(0,t)=0&\\
\psi_k(D/2,t)&=-k.
\end{aligned}\right\}
\end{equation}
The solution $\psi_k$ remains between the supersolution $\psi_{k,0}$ and the subsolution $(\log\tilde\phi_{0,1/k})'$, and is decreasing in $t$ and in $k$.  The inequality $\frac{\partial\psi_k}{\partial t}\leq 0$ implies bounds on $\psi_k'$ (since the barriers bound $\psi_k'$ at the endpoints), and it follows that $\psi_k$ exists and is Lipschitz in the first argument and continuous on $[0,D/2]\times\RR_+$, and smooth on $[0,D/2]\times(0,\infty)$.  By the strong maximum principle, $\psi_k(z,t)$ is strictly decreasing in $t$ on $(0,D/2)\times\RR_+$ unless $\psi_k=(\log\tilde\phi_{0,1/k})'$.  It follows that for each $k$, $\psi_k(.,t)$ converges in $C^\infty$ to $(
\log\tilde\phi_{0,1/k})'$ as $t\to\infty$.  By Theorem \ref{thm:log conc}, $\psi_{k,t}$ is a modulus of concavity for $\log u(.,t)$ for each $t>0$ and each $k\in\NN$, and therefore $\psi(.,t)=\lim_{k\to\infty}\psi_k(.,t)=\min_{k\in\NN}\psi_k(.,t)$ is also.  Since $\psi_k$ is bounded independent of $k$ in $C^q$ for every $q$ on compact subsets of $[0,D/2)\times(0,\infty)$, $\psi$ is a smooth solution of 
\begin{equation}\label{eq:dtpsiinfty}
\left.\begin{aligned}
\frac{\partial\psi}{\partial t}= \psi''+2\psi\psi'-\tilde V'&\quad\text{on\ }[0,D/2)\times(0,\infty)\\
\psi(0,t)=0&\\
\lim_{z\to D/2}\psi(z,t)&=-\infty,
\end{aligned}\right\}
\end{equation}
and $\psi(.,t)$ converges in $C^\infty$ to $(\log\tilde\phi_0)'$ as $t\to\infty$ on compact subsets of $[0,D/2)$.  This completes the proof.
\end{proof}

\begin{proof}[Proof of Theorem \ref{log concavity}]
Theorem \ref{log concavity} follows by applying Corollary \ref{cor:lim.beh} to the particular solution $u(x,t) = \phi_0(x)\E^{-\lambda_1t}$ of \eqref{eq:Par.Schrod}.  In this case 
$\nabla\log u$ is independent of $t$.  Therefore $\psi(.,t)$ is a modulus of concavity for $\log u_0$ for every $t$, and hence $(\log\tilde\phi_0)'=\inf_{t\geq 0}\psi(.,t)$ is also.
\end{proof}

\section{Examples and extensions} \label{examples section}

When $V$ is convex, 
Theorem \ref{spectral gap} gives Corollary \ref{convex corollary}, since then $0$ is a modulus of convexity for $V$, so we can choose $\tilde V=0$ and the first two eigenfunctions are $\tilde\phi_0(z)=\cos\left(\frac{\pi z}{D}\right)$ with $\mu_0=\frac{\pi^2}{D^2}$, and 
$\tilde\phi_1(z)=\cos\left(\frac{2\pi z}{D}\right)$ with $\mu_1=\frac{4\pi^2}{D^2}$, and hence
$\lambda_1-\lambda_0\geq \mu_1-\mu_0=\frac{3\pi^2}{D^3}$.

A positive lower bound on the Hessian of the potential implies stronger estimates on the gap since a particular modulus of convexity follows:
\begin{align*}
\left(DV(y)-DV(x)\right)\cdot\frac{y-x}{|y-x|} 
&= \frac{1}{|y-x|}\int_0^1 D^2V\big|_{(1-s)x+sy}(y-x,y-x)\,ds\\
&\geq K|y-x|=2\tilde V'\left(\frac{|y-x|}{2}\right)
\end{align*}
where $\tilde V(z) = \frac{K}2z^2$ is the harmonic oscillator potential.   Thus the fundamental gap on $\Omega$ is bounded below by that of the harmonic oscillator on the interval $[-D/2,D/2]$.  This provides a sharp version of \cite{yau-gap}*{Theorem 1.2}.
The result also extends to entire domains:  Suppose $V$ is a potential on $\RR^n$ with $D^2V\geq K\delta_{ij}$.  Then the eigenvalues $\lambda_i$ on $B_R(0)$ converge to those on $\RR^n$, so the fundamental gap for this potential on $\RR^n$ is no less than that of the one-dimensional harmonic oscillator, which is $\sqrt{2K}$.  Equality will hold in the limit of potentials of the form $V=\frac{K}{2}\left(x_1^2+n\sum_{i=2}^nx_i^2\right)$ as $n\to\infty$.  

Our result also gives useful consequences for non-convex potentials:  Consider a one-dimensional potential $\tilde V$ which is even and has non-negative third derivatives.  A useful example is the double-well potential $\tilde V(z) = -az^2+bz^4$ with $a,b>0$.  Now consider higher-dimensional potentials of the form
$$
V(x) = \tilde V(|x|)+c\sum_{i=2}^n x_i^2.
$$
For large $c$ these are double well potentials which agree with $\tilde V$ along the $x_1$ axis.  One can check directly that $\tilde V'$ is a modulus of convexity for $V$ for any $c\geq 0$, and it follows that the fundamental gap for such a potential on a convex domain $\Omega$ of diameter $D$ is bounded below by that of $\tilde V$ on the interval $[-D/2,D/2]$.  Again, the result is sharp, with equality holding in the limit as $c\to\infty$ in the case where $\Omega$ contains the interval $(-D/2,D/2)$.

An extension of our methods yields sharp results for domains in constant curvature spaces, implying that the spectral gap for a convex domain with convex potential is at least as large as the gap for a one-dimensional comparison problem (Equation (1) in \cite{Kroger} with Dirichlet boundary conditions), with equality in the limit of domains which are small neighbourhoods of a geodesic segment.  We will provide the details elsewhere.
%% extension to constant curvature spaces

\end{document}